\documentclass[a4paper,12pt]{article}
\usepackage{amsmath,amssymb,mathrsfs,euscript,amsopn,graphicx}
\usepackage[usenames]{color}
\usepackage{colortbl}
\numberwithin{equation}{section}

\newtheorem{lemma}{Lemma}[section]
\newtheorem{theorem}{Theorem}[section]
\newtheorem{proposition}{Proposition}[section]
\newtheorem{definition}{Definition}[section]
\newtheorem{remark}{Remark}[section]
\newtheorem{corollary}{Corollary}[section]

\newenvironment{proof}{\smallskip\noindent{\bf Proof.}\rm}{\hspace*{\fill} $\Box$\medskip}

\newenvironment{proofTh0}{\smallskip\noindent{\bf Proof of Lemma~\ref{Th0}.}\rm}{\hspace*{\fill} $\Box$\medskip}
\newenvironment{proofReduction}{\smallskip\noindent{\bf Proof of Lemma~\ref{reduction}.}\rm}{\hspace*{\fill} $\Box$\medskip}
\newenvironment{proofProp13}{\smallskip\noindent{\bf Proof of Lemmas~\ref{prop1} and \ref{prop3}.}\rm}{\hspace*{\fill} $\Box$\medskip}
\newenvironment{proofTh1}{\smallskip\noindent{\bf Proof of Theorem~\ref{Th1}.}\rm}{\hspace*{\fill} $\Box$\medskip}
\newenvironment{proofTh2}{\smallskip\noindent{\bf Proof of Theorem~\ref{Th2}.}\rm}{\hspace*{\fill} $\Box$\medskip}
\newenvironment{proofTh3}{\smallskip\noindent{\bf Proof of Theorem~\ref{Th3}.}\rm}{\hspace*{\fill} $\Box$\medskip}

\renewcommand{\d}{\mathrm d}
\newcommand{\dx}{\mathrm d x}
\newcommand{\e}{\mathrm e}
\renewcommand{\i}{\mathrm i}
\renewcommand{\o}{\mathrm o}

\newcommand{\res}{\operatorname{res}}
\newcommand{\rank}{\operatorname{rank}}
\newcommand{\Ran}{\operatorname{Ran}}
\newcommand{\supp}{\operatorname{supp}}

\newcommand{\mathfrakS}{\mathpalette\bigmathfrakS\relax}
\newcommand{\bigmathfrakS}[2]{\scalebox{1.2}{$#1\mathfrak{s}$}}

\newcommand{\bigc}{\scalebox{1.2}{$c$}}
\newcommand{\bigs}{\scalebox{1.2}{$s$}}

\title{On inverse spectral problems for self-adjoint Dirac operators with general boundary conditions}

\author{
D.~V.~Puyda\thanks{\emph{Email addresses:} dpuyda@gmail.com (D.~V.~Puyda)}\\
\emph{Ivan Franko National University of Lviv}\\
\emph{1 Universytetska str., Lviv, 79000, Ukraine}
}

\date{}

\begin{document}

\maketitle

\begin{abstract}
We consider the self-adjoint Dirac operators on a finite interval with summable matrix-valued potentials and general boundary conditions. For such operators, we study the inverse problem of reconstructing the potential and the boundary conditions of the operator from its eigenvalues and suitably defined norming matrices.
\end{abstract}

\section{Introduction}

In this paper, we consider the self-adjoint Dirac operators on $(-1,1)$ generated by the differential expressions
\begin{equation}\nonumber
\mathfrak t_q:=\frac1\i\begin{pmatrix}I&0\\0&-I\end{pmatrix}\frac\d{\d x}+
\begin{pmatrix}0&q\\q^*&0\end{pmatrix}
\end{equation}
and general boundary conditions of the form
$$
Ay(-1)+By(1)=0.
$$
Here, $q$ is an $r\times r$ matrix-valued function with entries belonging to $L_p(-1,1)$, $p\in[1,\infty)$, called the \emph{potential} of the operator, $I$ is the $r\times r$ identity matrix, $A$ and $B$ are $2r\times2r$ matrices with complex entries such that the operator is self-adjoint. For such operators, we introduce the notion of the spectral data -- eigenvalues and suitably defined norming matrices. We then study the inverse problem of reconstructing the potential and the boundary conditions of the operator from its spectral data.

Inverse spectral problems for Dirac and Sturm--Liouville operators with matrix-valued potentials arise in many areas of modern physics and mathematics. For instance, the inverse problems for quantum graphs (see, e.g., \cite{kuchment}) in some cases can be reduced to the ones for the operators with matrix-valued potentials.
Among the recent investigations in the area of inverse problems for Dirac-type systems we mention, e.g., the ones by Albeverio, Hryniv and Mykytyuk \cite{MykDirac}, Gesztesy et al. \cite{ClarkGesz,KisMakGesz,GeszOpV}, Malamud et al. \cite{Malamud3,Malamud1,Malamud2}, Sakhnovich \cite{Sakh1,Sakh2}.
The inverse problem of reconstructing the skew-self-adjoint Dirac system with a rectangular matrix-valued potential from the Weyl function was recently solved in \cite{SakhRect}.
Problems similar to the ones considered in this paper were recently treated for Sturm-Liouville operators with matrix-valued potentials in \cite{Korot1,Korot2,sturm}.
A considerable contribution to the spectral theory of differential operators with matrix-valued potentials was made by Rofe-Beketov et al. (see, e.g., \cite{Rofe}).
The operators with general boundary conditions also remain an object of interest these days. For instance, the completeness problem of root functions of general boundary value problems for the first order Dirac-type systems was recently solved in \cite{MalamudCompl}. We refer the reader to the extensive reference lists in [1--7, 10--12, 16, 18--20] for further results on the subject.

Recently, the inverse problem of reconstructing the self-adjoint Dirac operators with some separated boundary conditions from eigenvalues and norming matrices was solved in \cite{dirac1} (for the operators with square-integrable matrix-valued potentials) and \cite{dirac2} (for the more general case of the operators with summable matrix-valued potentials). In the present paper, we extend the results of \cite{dirac2} in order to solve the inverse spectral problem for the operators with general (especially, non-separated) boundary conditions.

The approach consists in reducing the problem to the one for the operators with separated boundary conditions acting in the space $L_2((0,1),\mathbb C^{4r})$. We then develop the Krein accelerant method \cite{dirac1,sturm,dirac2} in order to solve the inverse spectral problem for the operators with general separated boundary conditions.
We show that the accelerants of such operators do not depend on boundary conditions and uniquely determine the potentials of the operators. The boundary conditions can be then reconstructed from the asymptotics of the spectral data.

The paper is organized as follows. In the following section, we give the precise setting of the problem. In Sect.~\ref{results}, we introduce the approach and formulate the main results. In Sect.~\ref{proofs}, we prove the main results of this paper.

\emph{Notations.} Throughout this paper, we write $\mathcal M_r$ for the set of all $r\times r$ matrices with complex entries and identify $\mathcal M_r$ with the Banach algebra of linear operators in $\mathbb C^r$ endowed with the standard norm. We write $I=I_r$ for the $r\times r$ identity matrix and $\mathcal U_r$ for the set of all unitary matrices $U\in\mathcal M_r$. For an arbitrary $p\in[1,\infty)$, we write $\mathcal Q_p:=L_p((-1,1),\mathcal M_r)$ for the set of all $r\times r$ matrix-valued functions with entries belonging to $L_p(-1,1)$ and endow $\mathcal Q_p$ with the norm
$$
\|q\|_{\mathcal Q_p}:=\left(\int_{-1}^1\|q(s)\|^p\ \d s\right)^{1/p},\qquad q\in\mathcal Q_p.
$$
Similarly, we set $\mathfrak Q_p:=L_p((0,1),\mathcal M_{2r})$. The superscript $\top$ designates the transposition of vectors and matrices, e.g., $(c_1,\,c_2)^\top$ is the column vector $\begin{pmatrix}c_1\\c_2\end{pmatrix}$.

\section{Setting of the problem}\label{setting}

Given an arbitrary $q\in\mathcal Q_p$, we set
\begin{equation}\nonumber
J:=\frac{1}{\i}\begin{pmatrix}I&0\\0&-I\end{pmatrix},\qquad
Q:=\begin{pmatrix}0&q\\q^*&0\end{pmatrix}
\end{equation}
and consider the differential expression
$$
\mathfrak t_q:=J\frac{\d}{\dx}+Q
$$
on the domain $D(\mathfrak t_q)=W_2^1((-1,1),\mathbb C^{2r})$, where $W_2^1$ is the Sobolev space (see Appendix~\ref{AppSpaces}).
In the Hilbert space $\mathcal H:=L_2((-1,1),\mathbb C^{2r})$, we introduce the \emph{maximal operator} $T_q$ by the formula $T_qy=\mathfrak t_q(y)$,
$$
D(T_q)=\{y\in D(\mathfrak t_q)\mid\mathfrak t_q(y)\in\mathcal H\}.
$$
The adjoint operator $T_q^0:={T_q}^*$ is the restriction of $T_q$ onto the domain
$$
D(T_q^0)=\{y\in D(T_q)\mid y(-1)=y(1)=0\}.
$$
By definition, the operator $T_q^0$ will be called the \emph{minimal} one.
The objects of our study are self-adjoint extensions of the minimal operator $T_q^0$.


It is known (see, e.g., \cite{kreinvolterra}) that every self-adjoint extension of the minimal operator $T_q^0$ is the restriction of the maximal operator $T_q$ onto the domain
\begin{equation}\label{TABdom}
D(T)=\{y\in D(T_q)\mid Ay(-1)+By(1)=0\},
\end{equation}
where $A,B\in\mathcal M_{2r}$ are such that
\begin{equation}\nonumber
\rank(A\ \ B)=2r,\qquad AJA^*=BJB^*.
\end{equation}
Evidently, the self-adjoint extensions of $T_q^0$ cannot be parameterized by the matrices $A$ and $B$ uniquely since different pairs $(A,B)$ may lead to the same self-adjoint extension. However, using the standard technique involving the concept of boundary triplets one can prove the following lemma providing a \emph{unique} characterization of all self-adjoint extensions of the minimal operator $T_q^0$:

\begin{lemma}\label{Th0}
A linear operator $T:\mathcal H\to\mathcal H$ is a self-adjoint extension of the minimal operator $T_q^0$ if and only if there exists a unitary matrix $U\in\mathcal U_{2r}$ such that $T$ is the restriction of the maximal operator $T_q$ onto the domain (\ref{TABdom}) with
\begin{equation}\label{AUBU}
A=A_U:=P_2+UP_1,\qquad B=B_U:=P_1+UP_2,
\end{equation}
where
\begin{equation}\nonumber
P_1:=\begin{pmatrix}I&0\\0&0\end{pmatrix},\qquad P_2:=\begin{pmatrix}0&0\\0&I\end{pmatrix}.
\end{equation}
\end{lemma}

\noindent

According to Lemma~\ref{Th0}, we can parameterize all self-adjoint extensions of the minimal operator $T_q^0$ by unitary matrices $U\in\mathcal U_{2r}$. For an arbitrary $U\in\mathcal U_{2r}$, we denote by $T_{q,U}$ the restriction of the maximal operator $T_q$ onto the domain (\ref{TABdom}) with $A=A_U$ and $B=B_U$ given by formula (\ref{AUBU}).
For the operators $T_{q,U}$, we introduce the notion of the \emph{spectral data} -- eigenvalues and suitably defined norming matrices.

More precisely, the spectrum of the operator $T_{q,U}$ is purely discrete and consists of countably many isolated real eigenvalues of finite multiplicity accumulating only at $+\infty$ and $-\infty$. Throughout this section, we denote by $\lambda_j:=\lambda_j(q,U)$, $j\in\mathbb Z$, the pairwise distinct eigenvalues of the operator $T_{q,U}$ labeled in increasing order so that $\lambda_0<0\le\lambda_1$.


In order to introduce the norming matrices of the operator $T_{q,U}$, it is convenient to use the constructive definition which is similar to the one suggested in \cite{Korot1}. For every $\lambda\in\mathbb C$, we denote by $Y_q(\cdot,\lambda)\in W_2^1((-1,1),\mathcal M_{2r})$ a $2r\times2r$ matrix-valued solution of the Cauchy problem
\begin{equation}\label{MainCauchyProbl}
J\frac\d{\d x}Y+QY=\lambda Y,\qquad Y(0,\lambda)=I_{2r},
\end{equation}
where $I_{2r}$ is the $2r\times2r$ identity matrix. For every $j\in\mathbb Z$, we set
$$
M_j:=\frac12\,\int_{-1}^1 Y_q(s,\lambda_j)^*Y_q(s,\lambda_j)\ \d s.
$$
It follows that for every $j\in\mathbb Z$, $M_j={M_j}^*>0$. We then denote by $P_j:\mathbb C^{2r}\to\mathbb C^{2r}$ the orthogonal projector onto $\mathcal E_j:=\ker [A_U Y_q(-1,\lambda_j)+B_U Y_q(1,\lambda_j)]$ and define the positive self-adjoint operator $B_j:\mathcal E_j\to\mathcal E_j$ by setting
\begin{equation}\nonumber
B_j:=(P_j M_j P_j)\big|_{\mathcal E_j}.
\end{equation}
\begin{definition}\label{SDTqU}
For every $j\in\mathbb Z$, we set
$$
A_j(q,U):={B_j}^{-1} P_j
$$
and call $A_j(q,U)$ the norming matrix of the operator $T_{q,U}$ corresponding to the eigenvalue $\lambda_j(q,U)$.
The sequence
$$
\mathfrak a_{q,U}:=((\lambda_j(q,U), A_j(q,U)))_{j\in\mathbb Z}
$$
will be called the spectral data of the operator $T_{q,U}$.
\end{definition}

\begin{remark}\label{NormMatrRem}\rm
It follows from the definition of the norming matrices that $A_j M_j A_j=A_j$ for every $j\in\mathbb Z$, $A_j:=A_j(q,U)$. This will play an important role in the proof of Lemma~\ref{reduction} below.
\end{remark}

For the operators $T_{q,U}$, we study the inverse problem of reconstructing the potential $q$ and the unitary matrix $U$ from the spectral data. We shall give a complete description of the class
\begin{equation}\label{Tsd}
\mathfrak A_p:=\{\mathfrak a_{q,U}\mid q\in\mathcal Q_p,\ U\in\mathcal U_{2r}\}
\end{equation}
of the spectral data, show that the spectral data of the operator $T_{q,U}$ determine the potential $q$ and the unitary matrix $U$ uniquely and suggest how to find $q$ and $U$ from the spectral data.

\section{The approach and the main results}\label{results}

Our approach consists in reducing the problem for the operators $T_{q,U}$ to the one for the operators with separated boundary conditions that we now introduce.

Let $V\in\mathfrak Q_p$ (see Notations) be an arbitrary $2r\times2r$ matrix-valued function with entries belonging to $L_p(0,1)$, $p\in[1,\infty)$. Set
\begin{equation}\label{theta1Q}
\boldsymbol J:=\frac1\i\begin{pmatrix}I_{2r}&0\\0&-I_{2r}\end{pmatrix},\qquad
\boldsymbol V:=\begin{pmatrix}0&V\\V^*&0\end{pmatrix}
\end{equation}
and consider the differential expression
$$
\mathfrakS_V:=\boldsymbol J\frac\d{\d x}+\boldsymbol V
$$
on the domain
$$
D(\mathfrakS_V)=\left\{f:=\begin{pmatrix}f_1\\f_2\end{pmatrix} \ \Bigg| \ f_1,f_2\in W_2^1((0,1),\mathbb C^{2r}) \right\}.
$$
In the Hilbert space $\mathbb H:=L_2((0,1),\mathbb C^{2d})$, $d:=2r$, we introduce the auxiliary operator $S_{V,U}$, where $U\in\mathcal U_{2r}$, by the formula $S_{V,U}f=\mathfrakS_V(f)$,
$$
D(S_{V,U})=\{f\in D(\mathfrakS_V)\mid \mathfrakS_V(f)\in\mathbb H,\
f_1(0)=f_2(0),\ f_1(1)=Uf_2(1)\}.
$$

As in the case of the operators $T_{q,U}$, the spectrum of the operator $S_{V,U}$ is purely discrete and consists of countably many isolated real eigenvalues of finite multiplicity accumulating only at $+\infty$ and $-\infty$. In what follows, we denote by $\zeta_j:=\zeta_j(V,U)$, $j\in\mathbb Z$, the pairwise distinct eigenvalues of the operator $S_{V,U}$ labeled in increasing order so that $\zeta_0<0\le\zeta_1$.

For the operator $S_{V,U}$, the notion of the \emph{Weyl--Titchmarsh function} can be defined as in \cite{ClarkGesz} (see (\ref{WeylPrecise}) below for a precise definition); the Weyl--Titchmarsh function of the operator $S_{V,U}$ is a $2r\times2r$ matrix-valued meromorphic Herglotz function and $\{\zeta_j\}_{j\in\mathbb Z}$ is the set of its poles. This allows us to introduce the spectral data of the operator $S_{V,U}$ as in \cite{dirac1,dirac2}:

\begin{definition}\label{SDSQU}
Let $M_{V,U}$ be the Weyl--Titchmarsh function of the operator $S_{V,U}$. For every $j\in\mathbb Z$, we set
$$
C_j(V,U):=-\underset{\zeta=\zeta_j(V,U)}\res M_{V,U}(\zeta)
$$
and call $C_j(V,U)$ the norming matrix of the operator $S_{V,U}$ corresponding to the eigenvalue $\zeta_j(V,U)$. The sequence
$$
\mathfrak b_{V,U}:=((\zeta_j(V,U),C_j(V,U)))_{j\in\mathbb Z}
$$
will be called the spectral data of the operator $S_{V,U}$. The $2r\times2r$ matrix-valued measure
$$
\mu_{V,U}:=\sum_{j\in\mathbb Z} C_j(V,U)\delta_{\zeta_j(V,U)},
$$
where $\delta_\zeta$ is the Dirac delta measure centered at the point $\zeta$, will be referred to as its spectral measure.
\end{definition}

\noindent
As in \cite{dirac1,dirac2}, it follows that for every $j\in\mathbb Z$, $C_j(V,U)\ge0$ and the rank of $C_j(V,U)$ equals the multiplicity of the eigenvalue $\zeta_j(V,U)$.

We now state a connection between the operators $T_{q,U}$ and $S_{V,U}$:

\begin{lemma}\label{reduction}
For an arbitrary $q\in\mathcal Q_p$ and $U\in\mathcal U_{2r}$, the operator $T_{q,U}$ is unitarily equivalent to the operator $S_{V,U}$, where
\begin{equation}\label{Q}
V(x)=\begin{pmatrix}0&q(x)\\q(-x)^*&0\end{pmatrix},\qquad x\in(0,1).
\end{equation}
Moreover, the spectral data of the operator $T_{q,U}$ coincide with the spectral data of the operator $S_{V,U}$ with $V$ given by formula (\ref{Q}).
\end{lemma}

It thus follows from Lemma~\ref{reduction} that every sequence $\mathfrak a\in\mathfrak A_p$ (see (\ref{Tsd})) is the spectral data of the operator $S_{V,U}$ with the potential $V$ of the form (\ref{Q}). We now extend the results of \cite{dirac2} in order to solve the inverse spectral problem for the operators $S_{V,U}$. For such operators, we shall give a complete description of the class
$$
\mathfrak B_p:=\{\mathfrak b_{V,U}\mid V\in\mathfrak Q_p,\ U\in\mathcal U_{2r}\}
$$
of the spectral data, show that the spectral data of the operator $S_{V,U}$ determine the potential $V$ and the unitary matrix $U$ uniquely and suggest how to find $V$ and $U$ from the spectral data.

\subsection{The inverse problem for the operators $S_{V,U}$}

In what follows, let
$$
\mathfrak a:=((\lambda_j,A_j))_{j\in\mathbb Z}
$$
stand for an \emph{arbitrary} sequence, where $(\lambda_j)_{j\in\mathbb Z}$ is a strictly increasing sequence of real numbers such that $\lambda_0<0\le\lambda_1$ and $A_j$, $j\in\mathbb Z$, are non-zero non-negative matrices in $\mathcal M_{2r}$.
We first give the necessary and sufficient conditions for a sequence $\mathfrak a$ to belong to the class $\mathfrak B_p$. In order to formulate these conditions, we need to introduce some preliminaries.

We start by describing the asymptotics of $(\lambda_j)_{j\in\mathbb Z}$ and $(A_j)_{j\in\mathbb Z}$. The description will be much clearer after the following remark:

\begin{remark}\rm
Let $U\in\mathcal U_{2r}$ and $\gamma_1<\gamma_2<\ldots<\gamma_s$ be real numbers from the interval $[0,\pi)$ such that $\e^{2\i\gamma_k}$, $k=1,\ldots,s$, are all distinct eigenvalues of $U$. Then all distinct eigenvalues of the free operator $S_{0,U}$ take the form
\begin{equation}\label{FreeLambda}
\zeta_{ns+k}^0=\gamma_k+\pi n,\qquad k\in\{1,\ldots,s\},\ n\in\mathbb Z.
\end{equation}
The norming matrix of $S_{0,U}$ corresponding to the eigenvalue $\zeta_{ns+k}^0$ appears to be the orthogonal projector onto $\ker(U-\e^{2\i\gamma_k}I_{2r})$.
\end{remark}

\begin{definition}
We say that a sequence $\mathfrak a$ satisfies the condition $(C_1)$ if:
\begin{itemize}
\item[$(i)$]there exist real numbers $\gamma_1<\gamma_2<\ldots<\gamma_s$ from the interval $[0,\pi)$ such that with the numbers $\zeta_m^0$ of (\ref{FreeLambda}), $m\in\mathbb Z$, it holds
    \begin{equation}\label{LambdaAsymp1}
    \sum_{\lambda_j\in\Delta_m}|\lambda_j-\zeta_m^0|=\o(1),\qquad |m|\to\infty,
    \end{equation}
    and
    \begin{equation}\label{LambdaAsymp2}
    \sup_{m\in\mathbb Z}\sum_{\lambda_j\in\Delta_m}1<\infty,
    \end{equation}
    where
    $$
    \Delta_m:=\left[ \frac{\zeta_{m-1}^0+\zeta_m^0}{2},\ \frac{\zeta_m^0+\zeta_{m+1}^0}{2} \right);
    $$
\item[$(ii)$]there exist pairwise orthogonal projectors $P_1^0,\ldots,P_s^0\in\mathcal M_{2r}$ such that $$\sum_{k=1}^s P_k^0=I_{2r}$$ and for every $k\in\{1,\ldots,s\}$,
    \begin{equation}\label{AlphaAsymp}
    \left\|P_k^0-\sum_{\lambda_j\in\Delta_{ns+k}} A_j\right\|=\o(1),\qquad |n|\to\infty.
    \end{equation}
\end{itemize}
\end{definition}

\noindent
For every sequence $\mathfrak a$ satisfying the condition $(C_1)$, we define the unitary matrix $U_{\mathfrak a}\in\mathcal U_{2r}$ by the formula
\begin{equation}\label{UaDef}
U_{\mathfrak a}:=\sum_{k=1}^s \e^{2\i\gamma_k} P_k^0.
\end{equation}

Next, we denote by $\mu^{\mathfrak a}$ the $2r\times2r$ matrix-valued measure given by the formula
\begin{equation}\label{MuAdef}
\mu^{\mathfrak a}:=\sum_{j\in\mathbb Z} A_j\delta_{\lambda_j}
\end{equation}
and associate with $\mu:=\mu^{\mathfrak a}$ the $\mathbb C^{2r}$-valued distribution defined via
$$
(\mu,f):=\int\limits_{\mathbb R} f \ \d\mu,\qquad f\in\mathcal S^{2r},
$$
where $\mathcal S^{2r}$ is the Schwartz space of rapidly decreasing $\mathbb C^{2r}$-valued functions (see Appendix~\ref{AppSpaces}). As in \cite{dirac2}, we introduce a Fourier-type transform of $\mu^{\mathfrak a}$:

\begin{definition}
For an arbitrary measure $\mu:=\mu^{\mathfrak a}$, we denote by $\widehat\mu$ the $\mathbb C^{2r}$-valued distribution given by the formula
$$
(\widehat\mu,f):=(\mu,\widehat f),\qquad f\in\mathcal S^{2r},
$$
where $\widehat f(\lambda):=\int_{-\infty}^\infty \e^{2\i\lambda s}f(s) \d s$, $\lambda\in\mathbb R$.
\end{definition}

For an arbitrary sequence $\mathfrak a$ satisfying the condition $(C_1)$, set $\mu:=\mu^{\mathfrak a}$ and let $H_\mu$ be the restriction of the distribution $\widehat\mu-\widehat{\mu_0}$ to the interval $(-1,1)$, i.e.,
\begin{equation}\label{HMuDef}
(H_\mu,f):=(\widehat\mu-\widehat{\mu_0},f),\qquad f\in\mathcal S^{2r},\ \supp f\subset(-1,1),
\end{equation}
where $\mu_0:=\mu_{0,U}$ is the spectral measure of the free operator $S_{0,U}$, $U:=U_{\mathfrak a}$. Then the following lemma gives the necessary and sufficient conditions for a sequence $\mathfrak a$ to belong to the class $\mathfrak B_p$:

\begin{lemma}\label{prop1}
A sequence $\mathfrak a$ belongs to the class $\mathfrak B_p$, $p\in[1,\infty)$, if and only if it satisfies the condition $(C_1)$ and
\begin{itemize}
\item[$(C_2)$]there exists $n_0\in\mathbb N$ such that for all natural $n>n_0$,
    $$
    \sum_{m=-ns+1}^{ns}\ \sum_{\lambda_j\in\Delta_m}\rank A_j=4nr;
    $$
\item[$(C_3)$]the system of functions $\{ \e^{\i\lambda_j t}v\mid j\in\mathbb Z,\ v\in\mathrm{Ran}\ A_j \}$ is complete in $L_2((-1,1),\mathbb C^{2r})$;
\item[$(C_4)$]the distribution $H_\mu$, where $\mu:=\mu^{\mathfrak a}$, belongs to $L_p((-1,1),\mathcal M_{2r})$.
\end{itemize}
\end{lemma}

By definition, every sequence $\mathfrak a\in\mathfrak B_p$ is the spectral data of some operator $S_{V,U}$. It turns out that the operator $S_{V,U}$ is determined by its spectral data uniquely:

\begin{lemma}\label{prop2}
For every $p\in[1,\infty)$, the mapping $\mathfrak Q_p\times\mathcal U_{2r}\owns(V,U)\mapsto\mathfrak b_{V,U}\in\mathfrak B_p$ is bijective.
\end{lemma}

\noindent
We then solve the inverse problem of finding the operator $S_{V,U}$ from its spectral data. As in \cite{dirac1,sturm,dirac2}, we base our procedure on Krein's accelerant method:

\begin{lemma}\label{prop3}
Let $\mathfrak a\in\mathfrak B_p$ be a putative spectral data of the operator $S_{V,U}$. Set $\mu:=\mu^{\mathfrak a}$ by formula (\ref{MuAdef}) and $H:=H_\mu$ by formula (\ref{HMuDef}). Then $H\in\mathfrak H_p$ (see Appendix~\ref{AppKrein}) and
$$
V=\Theta(H),\qquad U=U_{\mathfrak a},
$$
where $\Theta:\mathfrak H_p\to\mathfrak Q_p$ is the Krein mapping given by formula (\ref{ThetaDef}) and $U_{\mathfrak a}\in\mathcal U_{2r}$ is given by formula (\ref{UaDef}).
\end{lemma}

\noindent
The function $H:=H_\mu$, where $\mu:=\mu^{\mathfrak a}$ and $\mathfrak a\in\mathfrak B_p$ is the spectral data of the operator $S_{V,U}$, will be called the \emph{accelerant} of the operator $S_{V,U}$.

\subsection{The inverse problem for the operators $T_{q,U}$}

We now use the results of the previous subsection to solve the inverse spectral problem for the operators $T_{q,U}$.

Recall that by virtue of Lemma~\ref{prop3}, for every sequence $\mathfrak a$ satisfying the conditions $(C_1)-(C_4)$ from Lemma~\ref{prop1} the distribution $H:=H_\mu$, where $\mu:=\mu^{\mathfrak a}$, appears to be an accelerant and belongs to the class $\mathfrak H_p$.
Then the following theorem gives a complete description of the class $\mathfrak A_p$ of the spectral data of the operators $T_{q,U}$:

\begin{theorem}\label{Th1}
A sequence $\mathfrak a$ belongs to the class $\mathfrak A_p$, $p\in[1,\infty)$, if and only if it satisfies the conditions $(C_1)-(C_4)$ from Lemma~\ref{prop1} and
\begin{itemize}
\item[$(C_5)$]the function $V=\Theta(H)$, where $H:=H_\mu$, $\mu:=\mu^{\mathfrak a}$ and $\Theta:\mathfrak H_p\to\mathfrak Q_p$ is the Krein mapping given by formula (\ref{ThetaDef}), satisfies the anti-commutative relation
    $$
    V(x)J=-JV(x)
    $$
    a.e. on $(0,1)$.
\end{itemize}
\end{theorem}

By definition, every sequence $\mathfrak a\in\mathfrak A_p$ is the spectral data of some operator $T_{q,U}$. It turns out that the operator $T_{q,U}$ is determined by its spectral data uniquely:

\begin{theorem}\label{Th2}
For every $p\in[1,\infty)$, the mapping $\mathcal Q_p\times\mathcal U_{2r}\owns(q,U)\mapsto\mathfrak a_{q,U}\in\mathfrak A_p$ is bijective.
\end{theorem}

\noindent
We then solve the inverse problem of finding the operator $T_{q,U}$ from its spectral data:

\begin{theorem}\label{Th3}
Let $\mathfrak a\in\mathfrak A_p$ be a putative spectral data of the operator $T_{q,U}$. Set $\mu:=\mu^{\mathfrak a}$ by formula (\ref{MuAdef}), $H:=H_\mu$ by formula (\ref{HMuDef}) and $V:=\Theta(H)$. Then
\begin{equation}\label{qFromQ}
q(x)=\begin{cases}V_{12}(x),&x\in(0,1),\\
V_{21}(-x)^*,&x\in(-1,0),\end{cases}
\end{equation}
where $V=(V_{ij})_{i,j=1}^2$, and $U=U_{\mathfrak a}$, where $U_{\mathfrak a}\in\mathcal U_{2r}$ is given by formula (\ref{UaDef})
\end{theorem}

The procedure of finding the operator $T_{q,U}$ from its spectral data can be visualized by means of the following diagram:
$$
{\mathfrak A_p
\owns
\mathfrak a
\overset{(3.7)}{\underset{s_1}\longrightarrow}
U=U_{\mathfrak a};}
$$
$$
{\mathfrak A_p
\owns
\mathfrak a
\overset{(3.8)}{\underset{s_2}\longrightarrow}
\mu:=\mu^{\mathfrak a}
\overset{(3.9)}{\underset{s_3}\longrightarrow}
H:=H_\mu
\overset{(B.2)}{\underset{s_4}\longrightarrow}
V:=\Theta(H)
\overset{(3.10)}{\underset{s_5}\longrightarrow}
q.}
$$
Here, $s_j$ denotes the step number $j$. The steps $s_1$, $s_2$, $s_3$ and $s_5$ are trivial. The basic and non-trivial step is $s_4$ which requires solving the Krein equation (\ref{KreinEq}).

\begin{remark}\rm
By virtue of the condition $(C_5)$, the description of the class $\mathfrak A_p$ is not formulated in terms of eigenvalues and norming matrices directly. Unfortunately, this condition cannot be easily formulated even in terms of the accelerant $H:=H_\mu$. However, a certain complication as compared to the case of the separated boundary conditions is naturally expected. For instance, recall the classical results \cite{Marchenko} on the inverse problem of reconstructing Sturm--Liouville operators from two spectra: therein, a description of the two spectra in the case of the periodic/antiperiodic boundary conditions appears to be much more complicated than the one for the operators with separated ones.
\end{remark}

\begin{remark}\label{AlternateRem}\rm
We define the norming matrices of the operator $T_{q,U}$ using Definition~\ref{SDTqU} and the norming matrices of the operator $S_{V,U}$ using Definition~\ref{SDSQU}. However, one can also define the norming matrices of the operator $S_{V,U}$ similarly as in Definition~\ref{SDTqU}. Namely, let $\boldsymbol Y_V(\cdot,\zeta)\in W_2^1((0,1),\mathcal M_{4r})$ be a $4r\times4r$ matrix-valued solution of the Cauchy problem
\begin{equation}\label{DoubleCauchyPr}
\boldsymbol J\frac\d{\d x}\boldsymbol Y+\boldsymbol V\boldsymbol Y=\zeta\boldsymbol Y,\qquad
\boldsymbol Y(0,\zeta)=I_{4r}.
\end{equation}
For every $j\in\mathbb Z$, set
$$
\boldsymbol M_j:=\frac12 \int_0^1 \boldsymbol Y_V(s,\zeta_j)^*\boldsymbol Y_V(s,\zeta_j)\ \d s, \qquad \zeta_j:=\zeta_j(V,U).
$$
Observe that $\boldsymbol M_j={\boldsymbol M_j}^*>0$ and denote by $\boldsymbol P_j:\mathbb C^{4r}\to\mathbb C^{4r}$ the orthogonal projector onto $\pmb{\mathcal E}_j:=\ker[\boldsymbol A\boldsymbol Y_V(0,\zeta_j)+\boldsymbol B\boldsymbol Y_V(1,\zeta_j)]$, where
$$
\boldsymbol A:=\frac1{\sqrt2} \begin{pmatrix}0&0\\-I&I\end{pmatrix},\qquad
\boldsymbol B:=\frac1{\sqrt2} \begin{pmatrix}I&-U\\0&0\end{pmatrix},\qquad I:=I_{2r}.
$$
Next, define the positive self-adjoint operator $\boldsymbol B_j:\pmb{\mathcal E}_j\to\pmb{\mathcal E}_j$ via $\boldsymbol B_j:=(\boldsymbol P_j\boldsymbol M_j\boldsymbol P_j)\big|_{\pmb{\mathcal E}_j}$ and the operator $\boldsymbol D_j:\mathbb C^{4r}\to\mathbb C^{4r}$ by setting $\boldsymbol D_j:={\boldsymbol B_j}^{-1}\boldsymbol P_j$. As in Definition~\ref{SDTqU}, one may call $\boldsymbol D_j$ the norming matrix of the operator $S_{V,U}$. However, it turns out that $\boldsymbol D_j$ is of the same rank as the norming matrix $C_j$ from Definition~\ref{SDSQU} and, moreover, there are simple formulas relating $C_j$ and $\boldsymbol D_j$:
\begin{equation}\label{AlternateDef}
C_j=-\frac12 a\boldsymbol J\boldsymbol D_j\boldsymbol Ja^*,\qquad
\boldsymbol D_j=-2\,\boldsymbol Ja^* C_j a\boldsymbol J,
\end{equation}
where
\begin{equation}\label{aDef}
a:=\frac1{\sqrt2}\begin{pmatrix}I,&-I\end{pmatrix},\qquad I:=I_{2r}.
\end{equation}
Thus there is no essential difference between defining the norming matrices of $S_{V,U}$ as described in Definition~\ref{SDSQU} or as described in this remark. However, using Definition~\ref{SDSQU} is more convenient in this paper. The proof of formulas~(\ref{AlternateDef}) will follow from the proof of Lemma~\ref{reduction} below.
\end{remark}

\section{Proofs}\label{proofs}

In this section, we prove the main results of this paper.

\subsection{Proof of Lemma~\ref{Th0}}

Let $q\in\mathcal Q_p$, $p\in[1,\infty)$. We start by proving Lemma~\ref{Th0} providing a parametrization of all self-adjoint extensions of the minimal operator $T_q^0$ (see Sect.~\ref{setting}). Although the proof essentially uses the concept of boundary triplets, we omit the terminology and reduce it to straightforward manipulations:

\begin{proofTh0}
Let the operators
$$
E_j:W_2^1((-1,1),\mathbb C^{2r})\to \mathbb C^{2r},\qquad j=1,2
$$
act by the formulae
$$
E_1f:=(f_1(1),\,f_2(-1))^\top,\qquad E_2f:=(f_1(-1),\,f_2(1))^\top,
$$
where $f_1$ and $f_2$ are $\mathbb C^r$-valued functions composed of the first $r$ and the last $r$ components of $f$, respectively. Define the operator $E:W_2^1((-1,1),\mathbb C^{2r})\to\mathcal G:=\mathbb C^{2r}\times\mathbb C^{2r}$ by the formula
$$
E f:=(E_1f,\,E_2f)^\top.
$$
Then a direct verification shows that for every $f,h\in D(T_q)$,
\begin{equation}\label{aux1}
(T_qf|h)_{\mathcal H}-(f|T_qh)_{\mathcal H}=
-\i(E_1f|E_1h)_{\mathbb C^{2r}}+\i(E_2f|E_2h)_{\mathbb C^{2r}}=(\boldsymbol JE f|E h)_{\mathcal G},
\end{equation}
where $T_q$ is the maximal operator (see Sect.~\ref{setting}) and $\boldsymbol J:=-\i\,\mathrm{diag}(I_{2r},\,-I_{2r})$.

Now let $\mathcal T$ stand for the set of all linear operators $T:\mathcal H\to\mathcal H$ such that $T_q^0\subset T\subset T_q$ and denote by $\mathcal T_s$ the set of all \emph{self-adjoint} operators $T\in\mathcal T$. For every $T\in\mathcal T$, set
$$
F_T:=\{Ef\mid f\in D(T)\}.
$$
It is then easily seen from (\ref{aux1}) that every operator $T\in\mathcal T$ is related to its adjoint $T^*$ via
$$
F_{T^*}=(\boldsymbol JF_T)^\perp.
$$
Hence, the operator $T\in\mathcal T$ belongs to $\mathcal T_s$ if and only if
$$
\dim F_T=2r,\qquad \boldsymbol JF_T\perp F_T.
$$
It now follows from (\ref{aux1}) that for every $T\in\mathcal T_s$
and $f\in D(T)$, $\|E_1f\|=\|E_2f\|$, and thus we find that the
operator $T\in\mathcal T$ belongs to $\mathcal T_s$ if and only if
there exists a unitary matrix $U\in\mathcal U_{2r}$ such that
$$
D(T)=\{f\in D(T_q)\mid (E_1+UE_2)f=0\},
$$
i.e. $D(T)=\ker(E_1+UE_2)$.
Finally, to complete the proof it only remains to observe that for an arbitrary $f\in W_2^1((-1,1),\mathbb C^{2r})$ one has $f\in\ker(E_1+UE_2)$ if and only if
$$
A_Uf(-1)+B_Uf(1)=0,
$$
where $A_U$ and $B_U$ are given by formula (\ref{AUBU}).
\end{proofTh0}

\subsection{Basic properties of the operators $S_{V,U}$}

Before proving Lemma~\ref{reduction} allowing us to reduce the inverse problem for the operators $T_{q,U}$ to the one for the operators $S_{V,U}$, we need to list some properties of the latter.

We start by introducing the Weyl--Titchmarsh function of the operator $S_{V,U}$ (see \cite{ClarkGesz}). Let $V\in\mathfrak Q_p$. For an arbitrary $\zeta\in\mathbb C$, let $\varphi_V(\cdot,\zeta)$ and $\psi_V(\cdot,\zeta)$ be a $4r\times2r$ matrix-valued solutions of the Cauchy problems
\begin{equation}\label{phiCauchyProbl}
\boldsymbol J\frac{\d}{\d x}\varphi+\boldsymbol V\varphi=\zeta\varphi,\qquad \varphi(0,\zeta)=\boldsymbol J a^*,
\end{equation}
and
$$
\boldsymbol J\frac{\d}{\d x}\psi+\boldsymbol V\psi=\zeta\psi,\qquad \psi(0,\zeta)=a^*,
$$
respectively, where $\boldsymbol J$ and $\boldsymbol V$ are given by formula (\ref{theta1Q}) and $a$ is given by formula~(\ref{aDef}).
Set $\bigc_{V,U}(\zeta):=b_U\psi_V(1,\zeta)$ and $\bigs_{V,U}(\zeta):=b_U\varphi_V(1,\zeta)$, where
\begin{equation}\label{bU}
b_U:=\frac1{\sqrt2}\begin{pmatrix}U^{-1/2},&-U^{1/2}\end{pmatrix}
\end{equation}
and the square root of $U$ is taken so that if $\e^{2\i\gamma_k}$, $\gamma_k\in[0,\pi)$, are all distinct eigenvalues of $U$, then $\e^{\i\gamma_k}$ are all distinct eigenvalues of $U^{1/2}$. Then the function
\begin{equation}\label{WeylPrecise}
M_{V,U}(\zeta):=-\bigs_{V,U}(\zeta)^{-1}\bigc_{V,U}(\zeta)
\end{equation}
will be called the \emph{Weyl--Titchmarsh function} of the operator $S_{V,U}$.

The following proposition is proved in \cite{dirac2}:

\begin{proposition}\label{SCLemma}
For every $V\in\mathfrak Q_p$, there exists a unique function $K_V\in G_p^+(\mathcal M_{4r})$ (see Appendix~\ref{AppSpaces}) such that for all $\zeta\in\mathbb C$ and $x\in[0,1]$,
\begin{equation}\label{phiTrasfOp}
\varphi_V(x,\zeta)=\varphi_0(x,\zeta)+\int_0^x K_V(x,s)\varphi_0(s,\zeta)\ \d s,
\end{equation}
where
$
\varphi_0(x,\zeta)=\frac1{\sqrt{2}\i}
\begin{pmatrix}\e^{\i\zeta x}I\\\e^{-\i\zeta x}I\end{pmatrix},
$
$I:=I_{2r}$, is a solution of (\ref{phiCauchyProbl}) in the free case $\boldsymbol V=0$.
\end{proposition}

For an arbitrary $\zeta\in\mathbb C$, we define the operator $\Phi_V(\zeta):\mathbb C^{2r}\to\mathbb H$ by setting
\begin{equation}\label{PhiOpDef}
[\Phi_V(\zeta)c](x):=\varphi_V(x,\zeta)c,\qquad x\in[0,1].
\end{equation}
It then follows from (\ref{phiTrasfOp}) that for every $\zeta\in\mathbb C$,
\begin{equation}\label{PhiTransfOp}
\Phi_V(\zeta)=(\mathscr I+\mathscr K_V)\Phi_0(\zeta),
\end{equation}
where $\mathscr K_V:\mathbb H\to\mathbb H$ is the integral operator with kernel $K_V$ and $\mathscr I$ is the identity operator in $\mathbb H$. Note that for every $V\in\mathfrak Q_p$, $\mathscr K_V$ is a Volterra operator so that $\mathscr I+\mathscr K_V$ is invertible.
Furthermore, it follows that for every $V\in\mathfrak Q_p$ and $\zeta\in\mathbb C$,
\begin{equation}\label{kerPhi}
\ker\Phi_V(\zeta)=\{0\},\qquad \Ran\Phi_V^*(\zeta)=\mathbb C^{2r}.
\end{equation}

\noindent
Now we are ready to state the basic properties of the operators $S_{V,U}$:

\begin{proposition}\label{DirPropTh}
For every $V\in\mathfrak Q_p$ and $U\in\mathcal U_{2r}$,
\begin{itemize}
\item[$(i)$]the operator $S_{V,U}$ is self-adjoint;
\item[$(ii)$]the spectrum $\sigma(S_{V,U})$ of the operator $S_{V,U}$ consists of countably many isolated real eigenvalues of finite multiplicity; moreover,
    $$
    \sigma(S_{V,U})=\{\zeta\in\mathbb C\mid \ker \bigs_{V,U}(\zeta)\neq\{0\}\};
    $$
\item[$(iii)$]for every $j\in\mathbb Z$, let $\mathscr P_j:\mathbb H\to\mathbb H$ be the orthogonal projector onto $\ker(S_{V,U}-\zeta_j\mathscr I)$, $\zeta_j:=\zeta_j(V,U)$ be eigenvalues of the operator $S_{V,U}$ and $C_j:=C_j(V,U)$ be the corresponding norming matrices; then for every $j\in\mathbb Z$ one has $C_j\ge0$ and
\begin{equation}\label{ProjForm}
\mathscr P_j=\Phi_V(\zeta_j)C_j\Phi_V^*(\zeta_j).
\end{equation}
\end{itemize}
\end{proposition}

\noindent
The proof of Proposition~\ref{DirPropTh} repeats the proof of Theorem~2.1 in \cite{dirac1}.

Since $S_{V,U}$ is a self-adjoint operator with discrete spectrum, it also follows that
\begin{equation}\label{ResIdent}
\sum_{j=-\infty}^\infty \mathscr P_j=\mathscr I;
\end{equation}
by virtue of the relations (\ref{ProjForm}) and (\ref{ResIdent}), the operators $\Phi_V(\zeta)$ will play an important role in this investigation.

\subsection{Proof of Lemma~\ref{reduction}}

Now we are ready to prove Lemma~\ref{reduction} allowing us to reduce the inverse spectral problem for the operators $T_{q,U}$ to the one for the operators $S_{V,U}$:

\begin{proofReduction}
Let $q\in\mathcal Q_p$, $U\in\mathcal U_{2r}$ and $\mathfrak a_{q,U}=((\lambda_j,A_j))_{j\in\mathbb Z}$ be the spectral data of the operator $T_{q,U}$. Consider the unitary transformation $\mathcal V:\mathcal H\to\mathbb H$ given by the formula
\begin{equation}\label{mathcalU}
(\mathcal Vy)(x)=\begin{pmatrix}y_1(x),&y_2(-x),&y_1(-x),&y_2(x)\end{pmatrix}^\top,\qquad x\in(0,1),
\end{equation}
where $y_1$ and $y_2$ are $\mathbb C^r$-valued functions composed of the first $r$ and the last $r$ components of $y$, respectively. Then a direct verification shows that
$$
T_{q,U}=\mathcal V^{-1}S_{V,U}\mathcal V,
$$
where the potential $V$ is given by formula (\ref{Q}). In particular, it then follows that the spectra of the operators $T_{q,U}$ and $S_{V,U}$ coincide and for every $j\in\mathbb Z$,
\begin{equation}\label{LambdaZeta}
\lambda_j=\zeta_j(V,U).
\end{equation}
Thus it only remains to prove that for every $j\in\mathbb Z$,
\begin{equation}\label{Aalpha}
A_j=C_j(V,U).
\end{equation}

For an arbitrary $\lambda\in\mathbb C$, define the operator $\Psi_q(\lambda):\mathbb C^{2r}\to\mathcal H$ by the formula
$$
[\Psi_q(\lambda)c](x):=\frac1{\sqrt2\i}Y_q(x,\lambda)c,\qquad x\in[-1,1],
$$
where $Y_q$ is a solution of the Cauchy problem (\ref{MainCauchyProbl}). For every $j\in\mathbb Z$, let $\mathcal P_j:\mathcal H\to\mathcal H$ be the orthogonal projector onto the eigenspace $\ker(T_{q,U}-\lambda_j\mathcal I)$, where $\mathcal I$ is the identity operator in $\mathcal H$. Then (\ref{Aalpha}) will be proved if we show that
\begin{equation}\label{PjForm}
\mathcal P_j=\Psi_q(\lambda_j)A_j\Psi_q^*(\lambda_j).
\end{equation}
Indeed, observe that for every $j\in\mathbb Z$,
\begin{equation}\label{PjRel}
\mathcal P_j=\mathcal V^{-1}\mathscr P_j\mathcal V,
\end{equation}
where $\mathscr P_j:\mathbb H\to\mathbb H$ is the orthogonal projector onto $\ker(S_{V,U}-\lambda_j\mathscr I)$ and $\mathcal V$ is the unitary transformation (\ref{mathcalU}). Furthermore, a direct verification shows that for every $\lambda\in\mathbb C$,
\begin{equation}\label{PhiPsiRel}
\Phi_V(\lambda)=\mathcal V\Psi_q(\lambda).
\end{equation}
We then obtain from (\ref{PjForm})--(\ref{PhiPsiRel}), (\ref{LambdaZeta}) and (\ref{ProjForm}) that for every $j\in\mathbb Z$,
$$
\Phi_V(\lambda_j)(A_j-C_j(V,U))\Phi_V^*(\lambda_j)=0.
$$
Since for every $\lambda\in\mathbb C$, $\ker\Phi_V(\lambda)=\{0\}$ and $\Ran\Phi_V^*(\lambda)=\mathbb C^{2r}$, this proves (\ref{Aalpha}).

Thus it only remains to prove (\ref{PjForm}).
For this purpose, note that for every $j\in\mathbb Z$, the operator $\tilde{\mathcal P}_j:=\Psi_q(\lambda_j)A_j\Psi_q^*(\lambda_j)$ is self-adjoint and
$$
\Ran\tilde{\mathcal P}_j=\Psi_q(\lambda_j)\mathcal E_j=\ker(T_{q,U}-\lambda_j\mathcal I),
$$
where $\mathcal E_j:=\ker [A_U Y_q(-1,\lambda_j)+B_U Y_q(1,\lambda_j)]$. Therefore, in order to prove that $\tilde{\mathcal P}_j=\mathcal P_j$ it suffices to prove that ${\tilde{\mathcal P}_j}^2=\tilde{\mathcal P}_j$. To this end, recall Remark~\ref{NormMatrRem} and verify that
$$
A_j\Psi_q^*(\lambda_j)\Psi_q(\lambda_j)A_j=
A_j\left\{\frac12\int_{-1}^1 Y_q(s,\lambda_j)^*Y_q(s,\lambda_j)\ \d s\right\}A_j=A_j M_j A_j
=A_j.
$$
Therefore, ${\tilde{\mathcal P}_j}^2=\tilde{\mathcal P}_j$ follows and the proof is complete.
\end{proofReduction}

The following important corollary now follows from Lemma~\ref{reduction}:

\begin{corollary}
Every sequence $\mathfrak a$ from the class $\mathfrak A_p$ belongs to the class $\mathfrak B_p$ and is the spectral data of the operator $S_{V,U}$ with the potential $V$ of the form (\ref{Q}).
\end{corollary}

\begin{remark}\rm
The proof of formulas~(\ref{AlternateDef}) providing a relations between differently defined norming matrices of the operator $S_{V,U}$ (see Remark~\ref{AlternateRem}) also follows from the proof of Lemma~\ref{reduction}. Namely, let $\boldsymbol Y_V(\cdot,\zeta)$ be a $4r\times4r$ matrix-valued solution of the Cauchy problem~(\ref{DoubleCauchyPr}) and $\pmb\Psi_V(\zeta):\mathbb C^{4r}\to\mathbb H$, where $\zeta\in\mathbb C$, be an operator defined by the formula
$$
[\pmb\Psi_V(\zeta)c](x):=\frac1{\sqrt2\i}\boldsymbol Y_V(x,\zeta)c,\qquad x\in[0,1].
$$
As in the proof of Lemma~\ref{reduction}, it can be shown that $\mathscr P_j=\pmb\Psi_V(\zeta_j)\boldsymbol D_j\pmb\Psi_V^*(\zeta_j)$, where $\mathscr P_j$ is the eigenprojector of the operator $S_{V,U}$ and $\boldsymbol D_j$ is as in Remark~\ref{AlternateRem}. Next, since $\varphi_V(x,\zeta)=\boldsymbol Y_V(x,\zeta)\boldsymbol Ja^*$ (see~(\ref{phiCauchyProbl})), it follows that $\Phi_V(\zeta)=\sqrt2\i\pmb\Psi_V(\zeta)\boldsymbol Ja^*$ (see~(\ref{PhiOpDef})); taking into account also (\ref{ProjForm}) we obtain that
$$
\pmb\Psi_V(\zeta_j)\left(
\boldsymbol D_j+2\boldsymbol Ja^* C_j a\boldsymbol J
\right)\pmb\Psi_V^*(\zeta_j)=0,
$$
where $C_j:=C_j(V,U)$ is as in Definition~\ref{SDSQU}. Since for an arbitrary $\zeta\in\mathbb C$ one has $\ker\pmb\Psi_V(\zeta)=\{0\}$ and $\Ran\pmb\Psi_V^*(\zeta)=\mathbb C^{4r}$, this proves the second relation in (\ref{AlternateDef}). The first one follows since $\boldsymbol J^2=-I_{4r}$ and $aa^*=I_{2r}$.
\end{remark}

\subsection{Proof of Lemmas~\ref{prop1}--\ref{prop3}}

We now proceed to solve the inverse spectral problem for the operators $S_{V,U}$.
The proof of Lemmas~\ref{prop1}--\ref{prop3} is based on the connection between the operators $S_{V,U}$ and $S_{V,I}$, where $I:=I_{2r}$ is the identity matrix. We then use the results of \cite{dirac2} where the direct and inverse spectral problems for the operators $S_{V,I}$ were solved.

Let $V\in\mathfrak Q_p$ and $U\in\mathcal U_{2r}$. We start from the following observation:

\begin{lemma}\label{AccLemma}
Let $\mathfrak a:=((\lambda_j,A_j))_{j\in\mathbb Z}$ be the spectral data of the operator $S_{V,U}$ and $\mu:=\mu^{\mathfrak a}$ be its spectral measure. Then for every $f\in\mathcal S^{2r}$ such that $\supp f\subset(-1,1)$,
$$
(H_\mu,f)=(H_\nu,f),
$$
where $\nu$ is the spectral measure of the operator $S_{V,I}$.
\end{lemma}

\begin{proof}
For every $j\in\mathbb Z$, let $\mathscr P_j:\mathbb H\to\mathbb H$ be the orthogonal projector onto the eigenspace $\ker(S_{V,U}-\lambda_j\mathscr I)$, where $\mathscr I$ is the identity operator in $\mathbb H$. We then find from (\ref{ProjForm}) and (\ref{ResIdent}) that
$$
\sum_{j=-\infty}^\infty \Phi_V(\lambda_j)A_j\Phi_V^*(\lambda_j)=\mathscr I,
$$
where the series on the left hand side converges in the strong operator topology.
Recalling also (\ref{PhiTransfOp}), we observe that
\begin{equation}\label{C4AuxEq1}
\sum_{j=-\infty}^\infty \Phi_0(\lambda_j)A_j\Phi_0^*(\lambda_j)=
(\mathscr I+\mathscr K_V)^{-1}(\mathscr I+{\mathscr K_V}^*)^{-1},
\end{equation}
where $\mathscr K_V:\mathbb H\to\mathbb H$ is the integral operator with kernel $K_V$ from Proposition~\ref{SCLemma}. Note that the right hand side of (\ref{C4AuxEq1}) depends \emph{only} on the potential $V$ of the operator $S_{V,U}$, while the left hand side of (\ref{C4AuxEq1}) depends \emph{only} on the spectral data.

Now let $\nu$ be the spectral measure of the operator $S_{V,I}$ and $H:=H_\nu$. It then follows from the results of \cite{dirac2} that $H\in L_p((-1,1),\mathcal M_{2r})$ and
\begin{equation}\label{C4AuxEq2}
(\mathscr I+\mathscr K_V)^{-1}(\mathscr I+{\mathscr K_V}^*)^{-1}=\mathscr I+\mathscr F_H,
\end{equation}
where $\mathscr F_H:\mathbb H\to\mathbb H$ is the integral operator with kernel
\begin{equation}\nonumber
F_H(x,t):=\frac{1}{2}\begin{pmatrix}H\left(\frac{x-t}{2}\right)&
H\left(\frac{x+t}{2}\right)\\
H\left(-\frac{x+t}{2}\right)&
H\left(-\frac{x-t}{2}\right)\end{pmatrix},
\qquad 0\le x,t\le1.
\end{equation}
Therefore, we find from (\ref{C4AuxEq1}) and (\ref{C4AuxEq2}) that
\begin{equation}\label{C4AuxEq3}
\sum_{j=-\infty}^\infty \Phi_0(\lambda_j)A_j\Phi_0^*(\lambda_j)=\mathscr I+\mathscr F_H,\qquad H:=H_\nu.
\end{equation}
The lemma will follow directly from this relation.

Indeed, set $\tilde{\mathcal H}:=L_2((0,1),\mathbb C^{2r})$ and consider the unitary transformation $\mathcal W:\tilde{\mathcal H}\to\mathbb H$ acting by the formula
$$
(\mathcal Wg)(x):=\frac1{\sqrt2}\begin{pmatrix}g\left(\frac{1+x}2\right),&
g\left(\frac{1-x}2\right)\end{pmatrix}^\top,\qquad g\in\tilde{\mathcal H}.
$$
Then a direct verification shows that $\mathscr F_H=\mathcal W\mathscr H\mathcal W^{-1}$, where $\mathscr H:\tilde{\mathcal H}\to\tilde{\mathcal H}$ is the integral operator given by the formula
$$
(\mathscr H g)(x)=\int_0^1 H(x-s)g(s)\ \d s.
$$
Furthermore, it also follows that $\Phi_0(\lambda)=\mathcal W\Upsilon_0(\lambda)$, where for an arbitrary $\lambda\in\mathbb C$, the operator $\Upsilon_0(\lambda):\mathbb C^{2r}\to\tilde{\mathcal H}$ acts by the formula
$$
[\Upsilon_0(\lambda)c](x):=\e^{2\i\lambda x}c.
$$
Therefore, (\ref{C4AuxEq3}) is reduced to the equality
\begin{equation}\label{C4AuxEq4}
\sum_{j=-\infty}^\infty \Upsilon_0(\lambda_j)A_j\Upsilon_0^*(\lambda_j)=\widetilde{\mathcal I}+\mathscr H,
\end{equation}
where $\widetilde{\mathcal I}$ is the identity operator in $\tilde{\mathcal H}$. In particular, in the free case $V=0$ (\ref{C4AuxEq4}) reads
\begin{equation}\label{C4AuxEq4a}
\sum_{n=-\infty}^\infty \Upsilon_0(\lambda_n^0)A_n^0\Upsilon_0^*(\lambda_n^0)=\widetilde{\mathcal I},
\end{equation}
where $\lambda_n^0$ and $A_n^0$, $n\in\mathbb Z$, are eigenvalues and norming matrices of the free operator $S_{0,U}$, respectively.
Since for an arbitrary $f\in\mathcal S^{2r}$ such that $\supp f\subset(-1,1)$ one has $(H,f)=\int_{-1}^1 H(s)f(s)\ \d s$,
$$
(\widehat\mu,f)=\sum_{j=-\infty}^\infty\
\int_{-1}^1 \e^{2\i\lambda_js}A_j\ f(s)\ \d s,
\qquad
(\widehat{\mu_0},f)=\sum_{n=-\infty}^\infty\
\int_{-1}^1 \e^{2\i\lambda_n^0s}A_n^0\ f(s)\ \d s,
$$
substituting (\ref{C4AuxEq4a}) into (\ref{C4AuxEq4}) and using the formulas for $\Upsilon_0(\lambda)$ and $\mathscr H$ one can easily find that
$$
(H_\mu,f):=(\widehat\mu-\widehat{\mu_0},f)=(H,f),
$$
as desired.
\end{proof}

We now use the results of \cite{dirac2} to obtain the following corollary:

\begin{corollary}\label{AccCor}
For an arbitrary $V\in\mathfrak Q_p$ and $U\in\mathcal U_{2r}$, the spectral data of the operator $S_{V,U}$ satisfy the conditions $(C_3)$ and $(C_4)$ from Lemma~\ref{prop1}.
\end{corollary}

\begin{proof}
It is proved in \cite{dirac2} that for an arbitrary $V\in\mathfrak Q_p$ it holds $H_\nu\in\mathfrak H_p$, where $\nu$ is the spectral measure of the operator $S_{V,I}$. It then follows from Lemma~\ref{AccLemma} that for an arbitrary $V\in\mathfrak Q_p$ and $U\in\mathcal U_{2r}$ it holds $H_\mu\in\mathfrak H_p$, where $\mu:=\mu^{\mathfrak a}$ is the spectral measure of the operator $S_{V,U}$ and $\mathfrak a:=((\lambda_j,A_j))_{j\in\mathbb Z}$ is its spectral data. Therefore, we immediately obtain that the spectral data of the operator $S_{V,U}$ satisfy the condition $(C_4)$.

In order to prove the condition $(C_3)$, observe that by virtue of (\ref{C4AuxEq3}) one has
$$
\ker(\mathscr I+\mathscr F_H)=
\ker \left(
\sum_{j=-\infty}^\infty \Phi_0(\lambda_j)A_j\Phi_0^*(\lambda_j)
\right)
=\bigcap\limits_{j=-\infty}^\infty \ker A_j\Phi_0^*(\lambda_j)
=\tilde{\mathcal W}\mathcal X^\bot,
$$
where $H:=H_\nu$, $\mathcal X:=\{ \e^{\i\lambda_j t}d\mid j\in\mathbb{Z},\ d\in\mathrm{Ran}\ A_j \}$ and $\tilde{\mathcal W}:L_2((-1,1),\mathbb C^{2r})\to\mathbb H$ is the unitary mapping acting by the formula $(\tilde{\mathcal W}f)(x)=(f(x),\ f(-x))^\top$, $x\in(0,1)$. Since $H\in\mathfrak H_p$, it follows from the results of \cite{dirac2} that $\mathscr I+\mathscr F_H>0$ and thus $\ker(\mathscr I+\mathscr F_H)=\{0\}$. Therefore, $\mathcal X^\bot=\{0\}$, which proves the condition $(C_3)$.
\end{proof}

\begin{remark}\label{AccRem}\rm
It is proved in \cite[Lemma~4.2]{dirac2} that for an \emph{arbitrary} sequence $\mathfrak a$ satisfying the conditions $(C_3)$ and $(C_4)$ one has $H_\mu\in\mathfrak H_p$, where $\mu:=\mu^{\mathfrak a}$.
\end{remark}

Next, since eigenvalues of the operator $S_{V,U}$ are zeros of the entire function $\tilde \bigs_{V,U}(\lambda):=\det \bigs_{V,U}(\lambda)$ (see Proposition~\ref{DirPropTh}), the standard technique based on Rouche's theorem implies that eigenvalues of $S_{V,U}$ satisfy the asymptotics (\ref{LambdaAsymp1}) and the condition (\ref{LambdaAsymp2}). Furthermore, since
$$
\|M_{V,U}(\lambda)-M_{0,U}(\lambda)\|=\o(1)
$$
as $\lambda\to\infty$ within the domain $\mathcal O_\varepsilon:=\{\lambda\in\mathbb C\mid \forall m\in\mathbb Z:\ |\lambda-\zeta_m^0|\ge\varepsilon\}$ for some $\varepsilon>0$, one can easily prove (\ref{AlphaAsymp}) and obtain that the spectral data of the operator $S_{V,U}$ satisfy the condition $(C_1)$.

Therefore, so far we proved that the spectral data of the operator $S_{V,U}$ satisfy the conditions $(C_1)$, $(C_3)$ and $(C_4)$.

For an arbitrary sequence $\mathfrak a:=((\lambda_j,A_j))_{j\in\mathbb Z}$ satisfying the conditions $(C_1)$, $(C_3)$ and $(C_4)$, we set $\mu:=\mu^{\mathfrak a}$, $H:=H_\mu$ and $V:=\Theta(H)$ (see Remark~\ref{AccRem}). For every $j\in\mathbb Z$, we then define the operator $\mathscr P_{\mathfrak a,j}:\mathbb H\to\mathbb H$ by the formula
\begin{equation}\label{tildeP}
\mathscr P_{\mathfrak a,j}:=\Phi_V(\lambda_j)A_j\Phi_V^*(\lambda_j).
\end{equation}

\begin{proposition}\label{B2prop}
Let $\mathfrak a$ be an arbitrary sequence satisfying the conditions $(C_1)$, $(C_3)$ and $(C_4)$. Then:
\begin{itemize}
\item[$(i)$]the series $\sum_{j\in\mathbb Z}\mathscr P_{\mathfrak a,j}$ converges to the identity operator $\mathbb H\to\mathbb H$ in the strong operator topology;
\item[$(ii)$]a sequence $\mathfrak a$ satisfies the condition $(C_2)$ if and only if $\{\mathscr P_{\mathfrak a,j}\}_{j\in\mathbb Z}$ is a system of pairwise orthogonal projectors in $\mathbb H$.
\end{itemize}
\end{proposition}

\noindent
The proof of Proposition~\ref{B2prop} can be obtained by a straightforward modification of the proof of Proposition~3.3 and Lemma~4.5 in \cite{dirac2}; the proof uses the factorization of integral operators and the vector analogue of Kadec's $1/4$-theorem.

We now use Proposition~\ref{B2prop} to prove Lemmas~\ref{prop1} and \ref{prop3}:

\begin{proofProp13}
Let $\mathfrak a\in\mathfrak B_p$ be the spectral data of the operator $S_{V,U}$. It then follows from the above that $\mathfrak a$ satisfies the conditions $(C_1)$, $(C_3)$ and $(C_4)$. Now observe that by virtue of Proposition~\ref{DirPropTh}, the operators $\mathscr P_{\mathfrak a,j}$, $j\in\mathbb Z$, coincide with eigenprojectors of the operator $S_{V,U}$. Proposition~\ref{B2prop} then implies that $\mathfrak a$ satisfies the condition $(C_2)$. This is the necessity part of Lemma~\ref{prop1}.

Now let $\mathfrak a:=((\lambda_j,A_j))_{j\in\mathbb Z}$ be an \emph{arbitrary} sequence satisfying the conditions $(C_1)-(C_4)$; set $\mu:=\mu^{\mathfrak a}$, $H:=H_\mu$ and $V:=\Theta(H)$. Define the operators $\mathscr P_{\mathfrak a,j}$, $j\in\mathbb Z$, by formula (\ref{tildeP}). It then follows from Proposition~\ref{B2prop} that $\{\mathscr P_{\mathfrak a,j}\}_{j\in\mathbb Z}$ is a complete system of pairwise orthogonal projectors in $\mathbb H$. Then the same arguments as in \cite{dirac2} will imply that $\mathfrak a$ coincides with the spectral data of the operator $S_{V,U}$ with $U:=U_{\mathfrak a}$.

Namely, let $\mathfrak b:=((\zeta_j,C_j))_{j\in\mathbb Z}$ be the spectral data of the operator $S_{V,U}$. As in \cite{dirac2}, we observe that it only suffices to prove the inclusion
\begin{equation}\label{Th4Emb}
\Ran\mathscr P_{\mathfrak a,j}\subset\ker(S_{V,U}-\lambda_j\mathscr I),\qquad j\in\mathbb Z.
\end{equation}
Indeed, taking into account completeness of $\{\mathscr P_{\mathfrak a,j}\}_{j\in\mathbb Z}$, we immediately conclude from (\ref{Th4Emb}) that $\lambda_j=\zeta_j$ for every $j\in\mathbb Z$. From this equality and from (\ref{Th4Emb}) we then obtain that for every $j\in\mathbb Z$, $\mathscr P_j-\mathscr P_{\mathfrak a,j}\ge0$, where $\mathscr P_j$ are eigenprojectors of the operator $S_{V,U}$. However, taking into account completeness of the systems $\{\mathscr P_{\mathfrak a,j}\}_{j\in\mathbb Z}$ and $\{\mathscr P_j\}_{j\in\mathbb Z}$, we observe that $\sum_{j\in\mathbb Z}(\mathscr P_j-\mathscr P_{\mathfrak a,j})=0$ and thus $\mathscr P_j-\mathscr P_{\mathfrak a,j}=0$ for every $j\in\mathbb Z$. Therefore, recalling the representation (\ref{ProjForm}) for $\mathscr P_j$, we find that
$$
\Phi_V(\lambda_j)\{C_j-A_j\}\Phi_V^*(\lambda_j)=0.
\qquad j\in\mathbb Z,
$$
Taking into account (\ref{kerPhi}) we then obtain that $A_j=C_j$. Together with $\lambda_j=\zeta_j$, this implies that $\mathfrak a=\mathfrak b$, as desired.

Thus it only remains to prove (\ref{Th4Emb}). Since $\mathscr P_{\mathfrak a,j}=\Phi_V(\lambda_j)A_j\Phi_V^*(\lambda_j)$ and $\Ran \Phi_V^*(\lambda)=\mathbb C^{2r}$ for an arbitrary $\lambda\in\mathbb C$, we find that for every $j\in\mathbb Z$,
$$
\Ran\mathscr P_{\mathfrak a,j}=\{\varphi_V(\cdot,\lambda_j)A_j c\mid c\in\mathbb C^{2r}\}.
$$
Since for every $\lambda\in\mathbb C$, $\varphi_V(\cdot,\lambda)$ is a solution of the Cauchy problem
\begin{equation}\label{AuxPhiCauchyProbl}
\boldsymbol J\tfrac{\d}{\dx}\varphi+\boldsymbol V\varphi=\lambda\varphi,\qquad \varphi(0,\lambda)=\boldsymbol Ja^*,
\end{equation}
we then find that for every $f\in\Ran\mathscr P_{\mathfrak a,j}$ it holds $\mathfrakS_V(f)=\lambda_j f$ and $f_1(0)=f_2(0)$. Therefore, it only remains to prove that for every $f\in\Ran\mathscr P_{\mathfrak a,j}$ one has $f_1(1)=U f_2(1)$ with $U:=U_{\mathfrak a}$. The latter reads that for every $i\in\mathbb Z$,
\begin{equation}\label{Th4AuxEq1}
b_U \varphi_V(1,\lambda_i)A_i=0,
\end{equation}
where $b_U$ is given by formula (\ref{bU}).

So let us prove (\ref{Th4AuxEq1}). To this end, recalling that $\varphi_V(\cdot,\lambda)$ is a solution of the Cauchy problem (\ref{AuxPhiCauchyProbl}) and integrating by parts, we obtain that for all $i,j\in\mathbb Z$ and $c,d\in\mathbb C^{2r}$,
$$
\lambda_i(\Phi_V(\lambda_i)c\ | \
\Phi_V(\lambda_j)d)=(\boldsymbol J\varphi_V(1,\lambda_i)c\ | \
\varphi_V(1,\lambda_j)d)+\lambda_j(\Phi_V(\lambda_i)c\ | \
\Phi_V(\lambda_j)d)
$$
and thus
\begin{equation}\label{Th3AuxEq1}
(\lambda_i-\lambda_j)\Phi_V(\lambda_j)^*\Phi_V(\lambda_i)
=\varphi_V(1,\lambda_j)^*\boldsymbol J\varphi_V(1,\lambda_i).
\end{equation}
Since $\mathscr P_{\mathfrak a,i}\mathscr P_{\mathfrak a,j}=0$ as $i\neq j$ and for all $\lambda\in\mathbb C$ it holds $\ker\Phi_V(\lambda)=\{0\}$ and $\Ran\Phi_V^*(\lambda)=\mathbb C^{2r}$, we find that $A_j\Phi_V^*(\lambda_j)\Phi_V(\lambda_i)A_i=0$, $i\neq j$. Therefore, we obtain from (\ref{Th3AuxEq1}) that
\begin{equation}\label{Th3AuxEq2}
A_j\varphi_V(1,\lambda_j)^*\boldsymbol J\varphi_V(1,\lambda_i)A_i=0,\qquad i\neq j.
\end{equation}

Let $j\in\mathbb Z$. Taking into account (\ref{Th3AuxEq2}), we find that
\begin{equation}\label{Th3AuxEq3}
\left\{\sum_{k=1}^s (-1)^{n+1} \sum_{\lambda_j\in\Delta_{ns+k}}\
\boldsymbol J\varphi_V(1,\lambda_j)A_j\right\}^*\varphi_V(1,\lambda_i)A_i=0
\end{equation}
for large values of $n\in\mathbb Z$. If we show that
\begin{equation}\label{Th3AuxEq4}
\lim\limits_{n\to\infty}\left\{\sum_{k=1}^s (-1)^{n+1} \sum_{\lambda_j\in\Delta_{ns+k}}\
\boldsymbol J\varphi_V(1,\lambda_j)A_j\right\}^*=b_U,
\end{equation}
then passing to the limit $n\to\infty$ in (\ref{Th3AuxEq3}) would yield (\ref{Th4AuxEq1}).
In order to prove (\ref{Th3AuxEq4}), taking into account the Riemann-Lebesgue lemma, (\ref{LambdaAsymp1}) and (\ref{AlphaAsymp}) we find that for every $k\in\{1,\ldots,s\}$,
$$
\lim\limits_{n\to\infty} \left\{ (-1)^{n+1}
\sum_{\lambda_j\in\Delta_{ns+k}}\ \boldsymbol J\varphi_V(1,\lambda_j)A_j \right\}^*
$$
$$
=\lim\limits_{n\to\infty} \left\{ (-1)^{n+1} \boldsymbol J\varphi_0(1,\zeta_{ns+k}^0)
\sum_{\lambda_j\in\Delta_{ns+k}} A_j\right\}^*
$$
$$
=\frac1{\sqrt2}
\left\{ \lim\limits_{n\to\infty} \sum_{\lambda_j\in\Delta_{ns+k}} A_j \right\}
\begin{pmatrix}\e^{-\i\gamma_k},&-\e^{\i\gamma_k}\end{pmatrix}
=
\frac1{\sqrt2}\begin{pmatrix}\e^{-\i\gamma_k}P_k^0,&-\e^{\i\gamma_k}P_k^0\end{pmatrix}.
$$
Since
$$
\sum_{k=1}^s \left\{\frac1{\sqrt2}
\begin{pmatrix}\e^{-\i\gamma_k}P_k^0,&-\e^{\i\gamma_k}P_k^0\end{pmatrix}\right\}
=\frac1{\sqrt2}\begin{pmatrix}U^{-1/2},&-U^{1/2}\end{pmatrix}=b_U,
$$
(\ref{Th3AuxEq4}) follows and thus we have proved that $\mathfrak a=\mathfrak b$, where $\mathfrak b$ is the spectral data of the operator $S_{V,U}$ with $V=\Theta(H)$ and $U=U_{\mathfrak a}$. This is the sufficiency part of Lemma~\ref{prop1} and the proof of Lemma~\ref{prop3}.
\end{proofProp13}

The proof of Lemma~\ref{prop2} repeats the proof of Theorem~1.3 in \cite{dirac2} and therefore we omit it in this paper.

\subsection{Proof of Theorems~\ref{Th1}--\ref{Th3}}

We now use Lemmas~\ref{reduction}--\ref{prop3} to prove Theorems~\ref{Th1}~--~\ref{Th3} and thus solve the inverse spectral problem for the operators $T_{q,U}$.

Let
$$
\mathscr T:=\{T_{q,U}:\mathcal H\to\mathcal H\mid q\in\mathcal Q_p,\ U\in\mathcal U_{2r}\},\quad
\mathscr S:=\{S_{V,U}:\mathbb H\to\mathbb H\mid V\in\mathfrak Q_p,\ U\in\mathcal U_{2r}\}.
$$
Recall that for every operator $T_{q,U}\in\mathscr T$ we introduce the \emph{associated operator} $S_{V,U}\in\mathscr S$ with the potential $V$ given by formula (\ref{Q}).
Taking into account that the mapping
$$
\mathcal Q_p\owns q\mapsto V(x):=\begin{pmatrix}0&q(x)\\q(-x)^*&0\end{pmatrix}\in
\{V\in\mathfrak Q_p\mid V(x)J=-JV(x)\ \text{a.e. on}\ (0,1)\}
$$
is bijective, we arrive at the following obvious remark:

\begin{remark}\label{AssocRem}
Let $S_{V,U}\in\mathscr S$. Then there exists an operator $T_{q,U}\in\mathscr T$ such that $S_{V,U}$ is associated to $T_{q,U}$ if and only if the potential $V$ of the operator $S_{V,U}$ satisfies the anti-commutative relation
$$
V(x)J=-JV(x)
$$
a.e. on $(0,1)$. In this case, such operator $T_{q,U}$ is unique and its potential $q$ can be found from $V$ by formula (\ref{qFromQ}).
\end{remark}

Now we are ready to prove Theorem~\ref{Th1} providing a complete description of the class $\mathfrak A_p$ of the spectral data of the operators $T_{q,U}$:

\begin{proofTh1}
\emph{Necessity.} Let $\mathfrak a\in\mathfrak A_p$ be the spectral data of the operator $T_{q,U}$ and $S_{V,U}$ be the associated operator. It then follows from Lemma~\ref{reduction} that $\mathfrak a$ is the spectral data of the operator $S_{V,U}$. From Lemmas~\ref{prop1} and \ref{prop3} we then obtain that $\mathfrak a$ satisfies the conditions $(C_1)-(C_4)$ and that $V=\Theta(H)$, where $H:=H_\mu$ and $\mu:=\mu^{\mathfrak a}$. Since the operator $S_{V,U}$ is associated to $T_{q,U}$, by virtue of Remark~\ref{AssocRem} we also obtain that $\mathfrak a$ satisfies the condition $(C_5)$.

\emph{Sufficiency.} Let $\mathfrak a$ be an arbitrary sequence satisfying the conditions $(C_1)-(C_5)$. From Lemmas~\ref{prop1}--\ref{prop3} we then obtain that $\mathfrak a$ is the spectral data of the unique operator $S_{V,U}$ with $V=\Theta(H)$ and $U=U_{\mathfrak a}$, where $H:=H_\mu$ and $\mu:=\mu^{\mathfrak a}$. Since $\mathfrak a$ satisfies the condition $(C_5)$, by virtue of Remark~\ref{AssocRem} we then obtain that there exists a unique operator $T_{q,U}$ such that $S_{V,U}$ is associated to $T_{q,U}$. By virtue of Lemma~\ref{reduction}, the spectral data of the operator $T_{q,U}$ coincide with $\mathfrak a$. This proves the sufficiency part of Theorem~\ref{Th1}.
\end{proofTh1}

Next, we prove Theorem~\ref{Th2} claiming that the spectral data of the operator $T_{q,U}$ determine the potential $q$ and the unitary matrix $U$ uniquely:

\begin{proofTh2}
Let $\mathfrak a\in\mathfrak A_p$ be the spectral data of the operator $T_{q,U}$ and $\tilde{\mathfrak a}\in\mathfrak A_p$ be the spectral data of the operator $T_{\tilde q,\tilde U}$. Assume that $\mathfrak a=\tilde{\mathfrak a}$. Let $S_{V,U}$ be the associated operator to $T_{q,U}$ and $S_{\tilde V,\tilde U}$ be the associated operator to $T_{\tilde q,\tilde U}$. It then follows from Lemma~\ref{reduction} that $\mathfrak a$ is the spectral data of $S_{V,U}$ and $\tilde{\mathfrak a}$ is the spectral data of $S_{\tilde V,\tilde U}$. Since $\mathfrak a=\tilde{\mathfrak a}$, it then follows from Lemma~\ref{prop2} that $S_{V,U}=S_{\tilde V,\tilde U}$. From Remark~\ref{AssocRem} we then obtain that $T_{q,U}=T_{\tilde q,\tilde U}$.
\end{proofTh2}

Finally, we prove Theorem~\ref{Th3} suggesting how to find the potential $q$ and the unitary matrix $U$ from the spectral data of the operator $T_{q,U}$:

\begin{proofTh3}
Let $\mathfrak a\in\mathfrak A_p$ be a putative spectral data of the operator $T_{q,U}$. It then follows from Lemma~\ref{reduction} that $\mathfrak a$ is the spectral data of the associated operator $S_{V,U}$. From Lemmas~\ref{prop1}--\ref{prop3} we then obtain that such operator $S_{V,U}$ is determined by its spectral data uniquely and that $V=\Theta(H)$ and $U=U_{\mathfrak a}$, where $H:=H_\mu$ and $\mu:=\mu^{\mathfrak a}$. Since $\mathfrak a$ satisfies the condition $(C_5)$, from Remark~\ref{AssocRem} we then obtain that there exists a \emph{unique} operator $T_{q,U}$ such that $S_{V,U}$ is associated to $T_{q,U}$ and that the potential $q$ of the operator $T_{q,U}$ can be found by formula (\ref{qFromQ}).
\end{proofTh3}

\subsection*{Acknowledgments} 
The author is grateful to his supervisor Yaroslav Mykytyuk for valuable ideas, immense patience and permanent attention to this work. The author would like to thank Rostyslav Hryniv for valuable suggestions in preparing this manuscript.

\appendix

\section{Spaces}\label{AppSpaces}

In this appendix, we introduce some spaces that are used in this paper.

For an arbitrary Banach space $X$, we denote by $L_p((a,b),X)$, $p\in[1,\infty)$, the Banach space of all strongly measurable functions $f:(a,b)\to X$ for which the norm
$$
\|f\|_{L_p}:=\left(\int_a^b \|f(t)\|_X^p\d t\right)^{1/p}
$$
is finite. We denote by $C^k([a,b],X)$ the Banach space of all $k$ times continuously differentiable functions $[a,b]\to X$ with the standard supremum norm. We write $W_p^1((a,b),X)$, $p\in[1,\infty)$, for the Sobolev space that is the completion of the linear space $C^1([a,b],X)$ by the norm
$$
\|f\|_{W_p^1}:=\left(\int_a^b \|f(t)\|_X^p\d t\right)^{1/p}
+\left(\int_a^b \|f'(t)\|_X^p\d t\right)^{1/p};
$$
every function $f\in W_p^1((a,b),X)$ has the derivative $f'$ belonging to $L_p((a,b),X)$.

As mentioned in \emph{Notations}, we write $\mathcal M_r$ for the Banach algebra of all $r\times r$ matrices with complex entries and identify it with the Banach algebra of all linear operators $\mathbb C^r\to\mathbb C^r$ endowed with the standard norm.

We denote by $G_p(\mathcal M_r)$, $p\in[1,\infty)$, the set of all measurable functions $K:[0,1]^2\to \mathcal M_r$ such that for all $x,t\in[0,1]$, the functions $K(x,\cdot)$ and $K(\cdot,t)$ belong to $L_p((0,1),\mathcal M_r)$ and, moreover, the mappings
$$
[0,1]\ni x\mapsto K(x,\cdot)\in L_p((0,1),\mathcal M_r),\qquad
[0,1]\ni t\mapsto K(\cdot,t)\in L_p((0,1),\mathcal M_r)
$$
are continuous. The set $G_p(\mathcal M_r)$ becomes a Banach space upon introducing the norm
$$
\|K\|_{G_p}=\max\left\{
\max\limits_{x\in[0,1]}\|K(x,\cdot)\|_{L_p},\
\max\limits_{t\in[0,1]}\|K(\cdot,t)\|_{L_p}\right\}.
$$
We denote by $G_p^+(\mathcal M_r)$ the set of all functions $K\in G_p(\mathcal M_r)$ such that $K(x,t)=0$ a.e. in $\Omega^-:=\{(x,t) \ | \ 0<x<t<1\}$.

Finally, we denote by $\mathcal S$ the Schwartz space of all smooth functions $f\in C^\infty(\mathbb R)$ whose derivatives (including the function itself) decay at infinity faster than any power of $|x|^{-1}$, i.e.
$$
\mathcal S:=\{f\in C^\infty(\mathbb R)\mid x^\alpha \mathrm D^\beta f(x)\to0\ as\ |x|\to\infty,\quad
\alpha,\beta\in\mathbb N\cup\{0\}\}.
$$
We set $\mathcal S^r:=\{(f_1,\ldots,f_r)^\top\mid f_j\in\mathcal S,\ j=1,\ldots,r\}$.

\section{The Krein accelerants}\label{AppKrein}

Here, we recall some facts concerning the notion of the Krein accelerants (see, e.g., \cite{dirac1,sturm,dirac2}).

\begin{definition}\label{AccDef}
A function $H\in L_1((-1,1),\mathcal M_r)$ is called an accelerant if $H(-x)=H(x)^*$ a.e. on $(-1,1)$ and for every $a\in(0,1]$, the integral equation
$$
f(x)+\int_0^a H(x-t)f(t)\ \d t=0,\qquad x\in(0,a),
$$
has only zero solution in $L_2((0,a),\mathbb C^r)$.
\end{definition}

\noindent
Throughout this paper, we denote by $\mathfrak H_p:=\mathfrak H_p(\mathcal M_{2r})$, $p\in[1,\infty)$, the set of all $2r\times2r$ matrix-valued accelerants belonging to the space $L_p((-1,1),\mathcal M_{2r})$; we endow $\mathfrak H_p$ with the metric of $L_p((-1,1),\mathcal M_{2r})$.

It is known (see, e.g., \cite{MykDirac}) that a function $H\in L_p((-1,1),\mathcal M_{2r})$ belongs to $\mathfrak H_p$ if and only if the Krein equation
\begin{equation}\label{KreinEq}
R(x,t)+H(x-t)+\int_0^x R(x,s)H(s-t)\ \d s=0,\qquad 0\le t\le x\le1,
\end{equation}
is solvable in $G_p^+(\mathcal M_{2r})$ (see Appendix~\ref{AppSpaces}). In this case, a solution of (\ref{KreinEq}) is unique and we denote it by $R_H(x,t)$. We then define the Krein mapping $\Theta:\mathfrak H_1\to\mathfrak Q_1$ (see \emph{Notations}) by the formula
\begin{equation}\label{ThetaDef}
[\Theta(H)](x):=\i R_H(x,0),\qquad x\in(0,1).
\end{equation}
It is proved in \cite{dirac2} that for every $p\in[1,\infty)$, the Krein mapping acts from $\mathfrak H_p$ to $\mathfrak Q_p$ and, moreover, appears to be a homeomorphism between $\mathfrak H_p$ and $\mathfrak Q_p$.


\begin{thebibliography}{1}

\bibitem{MykDirac}
{S.~Albeverio, R.~Hryniv and Ya.~Mykytyuk},	
\textit{Inverse spectral problems for {D}irac operators with summable potentials},	
{Russ. J. Math. Phys.}	
\textbf{12}	
{(2005)},	
{no.~4},	
{406--423}.	

\bibitem{Korot1}
{D.~Chelkak and E.~Korotyaev},	
\textit{Parametrization of the isospectral set for the vector-valued
{S}turm--{L}iouville problem},	
{J. Funct. Anal.}	
\textbf{241}	
{(2006)},	
{no.~1},	
{359--373}.	

\bibitem{Korot2}
{D.~Chelkak and E.~Korotyaev},	
\textit{Weyl--{T}itchmarsh functions of vector-valued {S}turm--{L}iouville operators on the unit interval},	
{J. Funct. Anal.}	
\textbf{257}	
{(2009)},	
{no.~5},	
{1546--1588}.	

\bibitem{ClarkGesz}
{S.~Clark and F.~Gesztesy},	
\textit{Weyl--{T}itchmarsh {M}-function asymptotics, local uniqueness results, trace formulas, and {B}org-type theorems for {D}irac operators},	
{Trans. Amer. Math. Soc.}	
\textbf{354}	
{(2002)},	
{no.~9},	
{3475--3534}.	

\bibitem{SakhRect}
{B.~Fritzsche, B.~Kirstein, I.~Ya.~Roitberg and A.~L.~Sakhnovich},	
\textit{Skew-Self-Adjoint Dirac System with a Rectangular Matrix Potential: Weyl Theory, Direct and Inverse Problems},	
{Integr. Equ. Oper. Theory}	
\textbf{74}	
{(2012)},	
{no.~2},	
{163--187}.	

\bibitem{KisMakGesz}
{F.~Gesztesy, A.~Kiselev and K.~A.~Makarov},	
\textit{Uniqueness results for matrix-valued {S}chr{\"o}dinger, {J}acobi, and {D}irac-type operators},	
{Math. Nachr.}	
\textbf{239/240}	
{(2002)},	
{no.~1},	
{103--145}.	

\bibitem{GeszOpV}
{F.~Gesztesy, R.~Weikard and M.~L.~Zinchenko},	
\textit{Initial value problems and Weyl--Titchmarsh theory for Schr\"odinger operators with operator-valued potentials},	
{Oper. Matrices}	
\textbf{7}	
{(2013)},	
{no.~2},	
{241--283}.	

\bibitem{kreinvolterra}
{I.~Gokhberg and M.~Krein},	
\textit{Theory of Volterra operators in Hilbert space and its applications},	
{Nauka},	
{Moscow},		
{1967}.		

\bibitem{kuchment}
{P.~Kuchment},	
\textit{Quantum graphs: an introduction and a brief survey},	
{Analysis on graphs and its applications. Proc. Sympos. Pure Math.}	
\textbf{77}	
{(2008)},	
{291--312}.	

\bibitem{Malamud3}
{M.~Lesch and M.~Malamud},	
\textit{The inverse spectral problem for first order systems on the half line},	
{Oper. Theory Adv. Appl.}	
\textbf{117}	
{(2000)},	
{199--238}.	

\bibitem{Malamud1}
{M.~M.~Malamud},	
\textit{Borg-type theorems for first-order systems on a finite interval},	
{Funct. Anal. Appl.}	
\textbf{33}	
{(1999)},	
{no.~1},	
{64--68}.	

\bibitem{Malamud2}
{M.~M.~Malamud},	
\textit{Uniqueness questions in inverse problems for systems of differential equations on a finite interval},	
{Trans. Moscow Math. Soc.}	
\textbf{60}	
{(1999)},	
{173--224}.	

\bibitem{MalamudCompl}
{M.~M.~Malamud and L.~L.~Oridoroga},	
\textit{On the completeness of root subspaces of boundary value problems for first order systems of ordinary differential equations},	
{J. Funct. Anal.}	
\textbf{263}	
{(2012)},	
{no.~7},	
{1939{--}1980}.	

\bibitem{Marchenko}
{V.~A.~Marchenko},	
\textit{Sturm--Liouville operators and applications},	
{Birkh\"auser},	
{Basel},		
{1967}.		

\bibitem{dirac1}
{Ya.~V.~Mykytyuk and D.~V.~Puyda},	
\textit{Inverse spectral problems for Dirac operators on a finite interval},	
{J. Math. Anal. Appl.}	
\textbf{386}	
{(2012)},	
{no.~1},	
{177--194}.	

\bibitem{sturm}
{Ya.~V.~Mykytyuk and N.~S.~Trush},	
\textit{Inverse spectral problems for {S}turm{--}{L}iouville operators with matrix-valued potentials},	
{Inverse Problems}	
\textbf{26}	
{(2010)},	
{no.~015009},	
{(36 p.)}	

\bibitem{dirac2}
{D.~V.~Puyda},	
\textit{Inverse spectral problems for Dirac operators with summable matrix-valued potentials},	
{Integr. Equ. Oper. Theory}	
\textbf{74}	
{(2012)},	
{no.~3},	
{417--450}.	

\bibitem{Rofe}
{F.~S.~Rofe-Beketov and A.~M.~Kholkin},	
\textit{Spectral analysis of differential operators. Interplay between spectral and oscillatory properties},	
{World Scientific},	
{Hackensack},		
{2005}.		

\bibitem{Sakh1}
{A.~Sakhnovich},	
\textit{Dirac type and canonical systems: spectral and Weyl--Titchmarsh matrix functions, direct and inverse problems},	
{Inverse Problems}	
\textbf{18}	
{(2002)},	
{no.~2},	
{331--448}.	

\bibitem{Sakh2}
{A.~Sakhnovich},	
\textit{Dirac type system on the axis: explicit formulae for matrix potentials with singularities and soliton-positon interactions},	
{Inverse Problems}	
\textbf{19}	
{(2003)},	
{no.~4},	
{845--854}.	

\end{thebibliography}
\end{document}